\documentclass[12pt,a4paper]{article}

\usepackage[dvips]{graphicx}
\usepackage{amssymb,latexsym}
\usepackage{amsmath}
\usepackage{mathrsfs}
\usepackage{psfrag}
\usepackage[all]{xy}
\usepackage{color}
\usepackage{verbatim}

\newtheorem{prop}{Proposition}[section]
\newtheorem{theo}[prop]{Theorem}

\newtheorem{defn}[prop]{Definition}
\newtheorem{rem}[prop]{Remark}

\newenvironment{proof}
{\begin{trivlist} \item[\hskip \labelsep {\bf Proof}\hspace*{3 mm}]}
	{\hfill$\Box$\end{trivlist}}
\newenvironment{acknow}
{\begin{trivlist} \item[\hskip \labelsep {\bf Acknowledgments.}]}
	{\end{trivlist}}

\def \BR {\mathbb R}

\def \BL {\mathbb L}

\def \BP {\mathbb P}
\def \BS {\mathbb S}

\def \a {\alpha}
\def \b {\beta}
\def \c {\gamma}

\begin{document}

\title{Binary differential equations associated to congruences of lines in Euclidean 3-space}
\author{J. W. Bruce and F. Tari}

\maketitle
\begin{abstract}
We study quotients of quadratic forms and associated polar lines in the projective plane. Our results, applied pointwise to quadratic differential forms, shed some light on classical  binary differential equations (BDEs) associated to congruences of lines in Euclidean 3-space and allows us to introduce a new one. 
The new BDE yields a new singular surface in the Euclidean 3-space associated to a congruence of lines.
We determine the generic local configurations of the above BDEs on congruences. 
\end{abstract}

\renewcommand{\thefootnote}{\fnsymbol{footnote}}

\section{Introduction} \label{sec:Int}

Let $\BL$ denote the space of lines in real Euclidean 3-space; the geometry of submanifolds of $\BL$ is well studied, with contributions from such luminaries as Hamilton, Dupin, Study and Blaschke. The set $\BL$ can be modelled by the tangent bundle to the unit sphere $T\BS^2\subset \BS^2\times \BR^3$, with a point $(n,x)$ corresponding to the line $x+tn$, although this does depend on an arbitrary choice of origin. In this paper we consider generic surfaces in $\BL$, known classically as congruences, and some associated BDEs. We classify the singularities of these BDEs, relating them to other aspects of the geometry. 

We shall suppose that the congruence $Z\subset \BL$ is a smooth regular surface. Classically, congruences were specified by a smooth immersion 
$$(n,x):U\to \BS^2\times \BR^3$$ 
with a point $(n(u,v),x(u,v))$ determining the line $x(u,v)+tn(u,v)$; $x(U)$ is called the director surface or {\it directrix}. The directrix is clearly not unique: if $f:U\to \BR$ is any smooth function the map $(n,x+fn)$ represents the same family of lines. In what follows we will only be interested in local properties of congruences so we are considering germs of mappings $(n,x):\BR^2,0\to \BS^2\times \BR^3$, though sometimes we specify a domain, that is an open set $U\subset \BR^2$. 

There are a range of ways to study the geometry of congruences, but because we need to carry out explicit calculations we prefer the Gaussian approach, employed for example in \cite{Weatherburn}, using quadratic forms, that is families of quadratic  differential forms $Q:T_zZ\to \BR$ which depends smoothly on $z\in Z$. Clearly a curve on $Z$ through a point $z$ yields the germ of a ruled surface in $\BR^3$. Some properties of these ruled surfaces depend only on the direction of the tangent to the curve and classically these were used, as below, to pick up distinct directions in $T_zZ$ and to associate to $Z$ some generally singular surfaces in $\BR^3$. For example consider a curve $\gamma=(n,x) : I\to Z\subset \BL$ with $0\in I$ an open interval and $\gamma(0)=(n(0),x(0) )= L$. We have an associated ruled surface parametrised by $(t,s)\mapsto x(t)+sn(t) \subset \BR^3$. 

On the line $L$ there is a {\it central point}, in antiquated terminology the foot of the common perpendicular of {\it consecutive} lines. These central points sweep out the {\it striction curve} of the ruled surface. We shall denote by $r=-(x'\cdot n'/n'\cdot n')(0)$ the signed distance from $x(0)$ to the central point in the $n(0)$ direction.
Clearly the central point only depends on 
the tangent to the curve at $z$.

There is also a classical invariant $\lambda=[x',n,n']/n'\cdot n'$, called the {\it distribution parameter} or {\it pitch} of the ruled surface, where $[-,-,-]$ denote the usual scalar triple product in $\mathbb R^3$. Clearly this only depends on the point $\gamma(0)$ and  $\gamma'(0)$. When this is identically zero we have a developable surface, generally the set of tangent lines to a smooth space curve $C$, and the line of striction here is the edge of regression, that is $C$ which does not depend on the choice of directrix.
We have the following concepts (see for example, \cite{Weatherburn}).

\begin{enumerate}
\item The directions in $T_zZ$  along which the values of $r$ are extreme are called {\it principal directions}; as we shall see generally there are two such directions. The corresponding central points on $L\in \BR^3$ are called {\it boundary points}. 

\item A point $z$ is called an {\it umbilic point} if every direction in $T_zZ$ is a principal direction. The integral curves on $Z$ of the principal directions are the {\it principal curves} and the surfaces swept out by the boundary points in $\BR^3$ are called the {\it boundary surfaces}. 

\item The surface in $\BR^3$ swept out by the midpoints of the boundary points is called the {\it middle surface}.

\item There are $0, 1$ or $2$ directions in $T_zZ$, called {\it torsal directions}, for which the pitch $\lambda$ of any associated ruled surface vanishes (infinitesimally we have a developable surface in $\mathbb R^3$). The associated central points on the ray $L\subset \mathbb R^3$ are called {\bf{\it foci}}. The point $z\in Z$ is called a {\it parabolic point} when the two torsal directions coincide. The surface in $\BR^3$ swept out by the foci is called a {\it focal surface}. At any point the midpoint of the two foci (which is always real even if the two focal points are imaginary, see \cite{BruceTAffineCong}) is also the midpoint of the boundary points. 

\item For each $z\in Z$ there are at most two directions in $T_zZ$, called {\it mean directions}, along which the values of the pitch is extreme. Their associated points on $L$ are precisely the midpoints of the focii (or boundary) points.
\end{enumerate}

The principal, torsal and mean directions are solutions of BDEs. These are equations of the form
\begin{equation}\label{eq:bde}
	{\bf a}(u,v)dv^2+2{\bf b}(u,v)dvdu+{\bf c}(u,v)du^2=0, 
\end{equation}
where ${\bf a},{\bf b},{\bf c}$ are smooth functions on some open set $U\subset \BR^2$. Clearly any quadratic differential form $\omega$ on $Z$ yields a BDE, given by $\omega=0$.

We establish the generic local configurations of the solutions of the above BDEs, and a new BDE associated to $Z$ called the characteristic BDE. The new BDE yields a new singular surface in the Euclidean 3-space associated to $Z$ (\S \ref{ssec:charc}).
In \S \ref{sec:BDEsCongruences} we see how the above BDEs are related using the approach in \cite{dupin}, where the elementary geometry of binary quadratic forms casts light on the geometry of BDEs (see Figure \ref{fig:pencilsQis}). The relations between the above BDEs 
can be deduced from a more general result on quotient of quadratic forms given in \S \ref{sec:quotients}.

\section{Preliminaries} \label{sec:prel}

We denote  the BDE  \eqref{eq:bde} by $\omega=({\bf a},2{\bf b},{\bf c})$ and refer to ${\bf a}, {\bf b}, {\bf c}$ as the coefficients of the BDE. Such an equation determines a pair of distinct directions at each point in the region where $\delta={\bf b}^2-{\bf ac}>0$ and no direction at points where $\delta<0$.  The set $\delta=0$ is {\it the discriminant} of the BDE.  If the coefficients of the BDE do not all vanish at a point $(u,v)$ on the discriminant, then it determines a unique (double) direction there; if they do, then all directions at $(u,v)$ are considered solutions.  

The solution curves of a BDE form a pair of transverse foliations in the region \mbox{$\delta>0$}. These foliations together with the discriminant curve constitute the so-called  {\it configuration} of the BDE. Two BDEs are said to be topologically equivalent if there is a local homeomorphism in the plane which maps the configuration of one equation to that of the other. 

Let $q$ be a regular point on the discriminant curve. If the unique solution of $\omega$ at $q$ is transverse to the discriminant curve,  then the local configuration of the BDE is a family of cusps (see \cite{davbook} for references and Figure \ref{fig:Dav}). When the direction is tangent to the discriminant curve, there are three generic (i.e., stable) local topological models: folded saddle, folded node and folded focus (\cite{davbook}, Figure \ref{fig:Dav}). If we write $p=dv/du$ and the 2-jet of the BDE in the form 
$$
a_0p^2+(b_0+b_1u+b_2v)p+c_1u+c_2v+c_3u^2+c_4uv+c_5v^2,
$$
then according to Lemma 2.1 in \cite{duality}, the origin is a fold singularity if $a_0\ne 0$ and $b_0=c_1=0$ (when $b_0=0$ and $a_0c_1\ne 0$ we get a family of cusps). At a folded singularity, setting $\lambda=(4a_0c_3- b_1^2-b_1c_2)/(4c_2^2)$, the singularity is of type folded saddle if $\lambda<0$, folded node if $0<\lambda<1/16$ and folded focus if $\lambda>1/16$.

\begin{figure}
\begin{center}
\includegraphics[width=13cm,height=2.8cm]{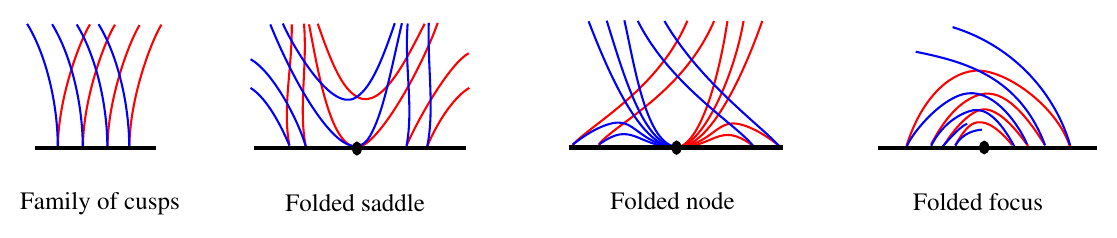}
\caption{Stable configurations of a BDE at points on its discriminant.}
\label{fig:Dav}
\end{center}
\end{figure}

At a point where all the coefficients of the BDE vanish, the discriminant is singular.  Our interest here is when the singularity is of type $A_1^+$ (i.e.,  $\delta\circ h^{-1}(u,v)=\pm (u^2+v^2)$ for some germ of a diffeomorphism $h$). 
Then the BDE has three generic possible configurations as shown in Figure \ref{fig:Darboux}, classified in (\cite{bruce-fidal,bdes, guinez}), called the star, monstar or lemon (here, by generic we mean within the set of BDEs with coefficients vanishing at the origin and whose discriminants have an $A_1^+$-singularity). If we write the 1-jet of the coefficients of the BDE at the origin in the form $j^1\omega=(a_1u+a_2v,2b_1u+2b_2v,c_1u+c_2v)$, then according to \cite{bdes}, at an $A_1^+$-singularity of the discriminant  we get a star or monstar (resp. lemon) if the cubic $\phi(p)=a_2p^3+(2b_2+a_1)p^2+(2b_1+c_2)p+c_1$ has three distinct roots  (resp. one root). The cases star and monstar are distinguished by the signs of $\alpha(p_i)\phi'(p_i)$ at the roots of $\phi$, where $\alpha(p)=a_2p^2+(b_2+a_1)p+b_1$. If all of them are positive we get a star, otherwise we get a monstar. The star, monstar and lemon singularities of BDEs are not stable within the set of all BDEs (\cite{Codim1bdes}). However, those that appear here associated to congruences are stable in that context.

\begin{figure}
	\begin{center}
		\includegraphics[width=12cm,height=3cm]{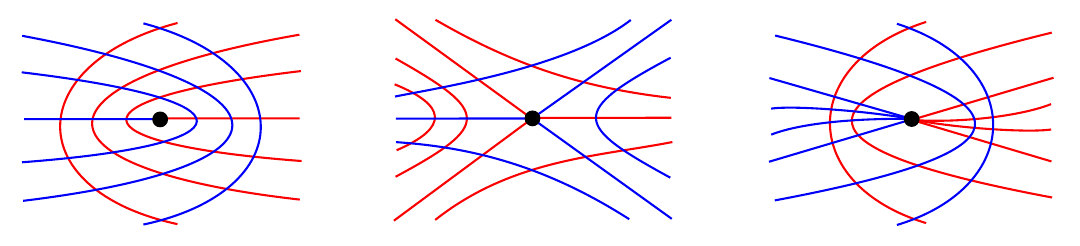}
		\caption{Generic configurations of BDEs of Morse Type $A_1^+$: lemon (left), star (centre), monstar (right).}
		\label{fig:Darboux}
	\end{center}
\end{figure}

We parametrise locally a congruence by $(n,x):U\to \BS^2\times \mathbb R^3$, where $U$ is an open set in $\mathbb R^2$ and identify germs of congruences with germs of mappings $C^{\infty}(U,\BR^5)$. 
We endow the set $C^{\infty}(U,\BR^5)$ with the Whitney $C^{\infty}$-topology and say that a property of congruences is {\it generic} if it is satisfied in an open and dense subset of $C^{\infty}(U,\BR^5)$.
To prove that a property is generic we consider the map  $\Phi: U\mapsto J^k(2,5)$ given $\Phi(u,v)=j^k_{(u,v)}(n,x)$. A given property is usually represented by a real algebraic variety $V$ in $J^k(2,5)$ with a smooth stratification. Thom's transversality theorem asserts that for generic congruences, $\Phi$ is transverse to the stratification of $V$. This means, in particular, that a property defined by more than three independent conditions, that is the codimension of $V>3$, is not generic.

\section{Quotients of quadratic forms} \label{sec:quotients}

Some key properties of the special directions on congruences of lines and their associated BDEs given in the introduction can be derived from a more general result below on quotients of quadratic forms.

\begin{theo}\label{theo:extrems}
\noindent {\rm(1)} Let $q_i=a_i\a ^2+2b_i\a\b+c_i\b^2, i=1,2,$ be two quadratic forms and consider the quotient 
$$
f(\a,\b)=\frac{q_2(\a,\b)}{q_1(\a,\b)},
$$
defined off the set $q_1(\a,\b)=0$. There are two local {\rm(}global if $q_1$ is positive definite{\rm)} extremal sets of points of the function $f$ which lie along the rays $(t\a,t\b)$ where $(\a,\b)$ are solutions of the quadratic equation
$$
\left|
\begin{array}{ccc}
\b^2 & -\a\b & \a^2\\
a_1 & b_1 &c_1\\
a_2 & b_2 & c_2
\end{array}\right|=0.
$$
 This expression is classically known as the Jacobian of $q_1,q_2$, written $Jac(q_1,q_2)$ that is the determinant of the $2\times 2$ matrix of partial derivatives of $q_1,q_2$. These directions coincide if and only if $q_1$ and $q_2$ have a common root. 

\noindent {\rm (2)}  Representing $q_i$ by symmetric matrices $E_i=\left(\begin{array}{cc}a_i&b_i\\b_i&c_i\end{array}\right)$ 
and writing $w=(\a,\b)^T$, the directions in $(1)$ are the solutions of $E_2 w=\mu E_1 w$. The corresponding values of $\mu$, denoted $\mu_1, \mu_2$, are real if $E_1$ is positive or negative definite. 

\noindent {\rm (3)} The  directions $(\a_1,\b_1), (\a_2,\b_2)$ corresponding to $\mu_1$, $\mu_2$ are orthogonal with respect to $q_2$. With respect to the same form they bisect the directions given by $q_1=0$ when these are real; this only occurs when $q_2$ is positive definite. 
\end{theo}

\begin{proof} (1) Taking an affine chart the extremal points of $f(\a,1)$ are the solutions of 
$$
(a_1\a+b_1)(a_2\a^2+2b_2\a+c_2)-(a_2\a+b_2)(a_1\a^2+2b_1\a+c_1)=0,
$$
equivalently, 
$$
(a_1b_2-b_1a_2)\a^2+(a_1c_2-c_1a_2)\a+(b_1c_2-c_1b_2)=0,
$$
which can be written in the determinantal form as stated. It is easy to see that the solutions coincide if and only if $q_1, q_2$ have a common root.

\noindent (2) It is easy to check that if $E_2 w=\mu E_1 w$ then $(\a, \b)$ is a \lq root\rq  of the quadratic form $Jac(q_1,q_2)=0$ and conversely. The second part follows from the usual argument: if $E_2w=\mu E_1 w$, then taking complex conjugates $E_2{\bar w}={\bar \mu} E_1 {\bar w}$ so $wE_2{\bar w}={\bar \mu}w E_1 {\bar w}$ and ${\bar w}E_2 w=\mu {\bar w}E_1 w$. So $(\mu-{\bar \mu})({\bar w}E_1 w)=0$ and the second factor is zero if and only if $w=0$. 

\noindent (3) If $q_2$ has rank $2$ we can assume it is $\pm (\a^2+\b^2$) or $\a\b$. In the first case we know we can reduce $q_1=a\a^2+c\b^2$, and $Jac(q_1,q_2)=4(a-c)\a\b$ the results easily follow. In the second case we can clearly suppose that $q_1=a\a^2+c\b^2$. Now $Jac(q_1,q_2)=2(a\a^2-c\b^2)$ and if $ac\ge 0$ clearly the two directions given by $Jac(q_1,q_2)=0$ are orthogonal with respect to $q_2$. Note that $q_1=0$ then has no roots.  
\end{proof}

There is a classical (elementary) geometry of quadratic forms used in the study of BDEs in \cite{dupin}. A non-zero quadratic form $a\alpha ^2+2b \alpha\beta+c\beta ^2$ can be represented by the point $q=(a:2b:c)$ in the projective plane $\BR \BP^2$. In $\BR \BP^2$ the set of singular quadratic forms is the conic $\Gamma=\{q: b^2-ac=0\}$. The {\it polar line} $\widehat{q}$ of a point $q$ with respect to $\Gamma$ is the line that contains all points $p$ such that $q$ and $p$ are harmonic conjugate points with respect to the intersection points $R_1$ and $R_2$ of the conic $\Gamma$ and a variable line through $q$. If the polar line $\widehat{q}$ meets $\Gamma$, then the tangents to $\Gamma$ at the points of intersection meet at $q$. A point $(a_1:2b_1:c_1)$ is on the polar line of a point $p=(a:2b:c)$ if and only if $2bb_1-ac_1-a_1c=0$. Three points in the projective plane are said to form a {\it self-polar triangle} if the polar line of any vertex of the triangle is the line through the remaining two points. In our case the points represent quadratic forms, so
any vertex of the self-polar triangle is the Jacobian of the remaining two vertices.

\section{BDEs on congruences}\label{sec:BDEsCongruences}
A  quadratic differential form $\omega(du,dv)={\bf a}(u,v)dv^2+2{\bf b}(u,v)dvdu+{\bf c}(u,v)du^2$ 
on a congruence $Z$ determines a BDE $\omega=0$ on $Z$. 
So the associated BDEs of two quadratic differential forms $\omega_1,\omega_2$ yields a Jacobian BDE $Jac(\omega_1,\omega_2)=0$; the formula gives a quadratic form at every point $(u,v)$.
This yields, at each point, the directions in which the quotient $\omega_2/\omega_1$ has extrema. Note that Theorem \ref{theo:extrems} shows that $Jac(\omega_1,\omega_2)=0$ has a single solution at points where the resultant of $\omega_1, \omega_2$ vanishes. Moreover the directions determined by $Jac(\omega_1,\omega_2)=0$ are orthogonal with respect to $\omega_1$ or $\omega_2$ and bisect the directions $\omega_1=0$ (resp. $\omega_2=0)$ with respect to the possibly indefinite metric $\omega_2$ (resp. $\omega_1$).  So for example if we consider a smooth surface in $\mathbb R^3$ with first fundamental form $\omega_1$ and second fundamental form $\omega_2$ then the Jacobian BDE $Jac(\omega_1,\omega_2)=0$ determines the  principle directions, which from above are orthogonal, with the asymptotic directions at hyperbolic points bisecting those directions.

For a congruence $Z\subset \BL$ there are three key quadratic forms $Q_1, Q_3, Q$ on the tangent spaces $T_zZ$, with $Q$ only defined up to a multiple of $Q_1$, from which two other (well-defined) quadratic forms can be constructed (see \cite{Weatherburn}). 

\begin{defn}\label{def:QisAndCoeff}
\noindent{\rm (1)} Define $Q_1, Q, Q_3:T_zZ\to \BR$ by 
$Q_1(du,dv)=|| n_u du+ n_v dv||^2$,
$Q(du,dv)=(x_u du+x_v dv)\cdot(n_u du +n_v dv)$, $Q_3(du,dv)=[x_u du+x_v dv,n_u du+n_v dv,n]$, so that
$$
\begin{array}{ccl}
Q_1(du,dv)&=&n_v\cdot n_v dv^2+2n_u\cdot n_vdudv+n_u\cdot n_udu^2,\\
Q(du,dv)&=&n_v\cdot x_v dv^2+(n_u\cdot x_v+n_v\cdot x_u)dudv+n_u\cdot x_udu^2,\\
Q_3(du,dv)&=&[x_v, n_v,n] dv^2+([x_u, n_v,n]+[x_v, n_u,n])dudv+[x_u, n_u,n]du^2.
\end{array}
$$ 
We write 
$$
\begin{array}{c}
A=n_u\cdot n_u, B=n_u\cdot n_v, C=n_v\cdot n_v,\\
a=-n_u\cdot x_u, b_1=-n_u\cdot x_v, b_2=-n_v\cdot x_u, c=-n_v\cdot x_v,\\
b=-\frac12(b_1+b_2), {\bar b}=-\frac12(b_1-b_2).
\end{array}
$$
\noindent {\rm (2)} Define $Q_2, Q_4:T_zZ\to \BR$ by 
$Q_2= Jac(Q,Q_1)$ and $Q_4=Jac(Q_3,Q_1)$. 
\end{defn}

At each point $z\in Z$, we shall write $Q_i$ for $Q_i(du,dv)$.

\begin{prop}\label{prop:BDEscong}
\noindent {\rm (1)} $Q_1, Q_2, Q_3, Q_4$ are all well defined.

\noindent {\rm (2)} The quadratic form $Q_2$ 
is given by 
$$
\begin{array}{rcl}
	Q_2&=&\left|
	\begin{array}{ccc}
		dv^2 & -dudv & du^2\\
		A& B &C\\
		a& b&c 
	\end{array}
	\right|\\
&&\\
&=&(Bc-Cb)dv^2+(Ac-Ca)dudv+(Ab-Ba)du^2,
\end{array}
$$
and the quadratic form $Q_3$ is a non-zero multiple of 
$$
\begin{array}{rcl}
Q_3&=&(Bc-Cb-C{\bar b})dv^2+(Ac-Ca-2B{\bar b})dudv+(Ab-Ba-A{\bar b})du^2\\
&=&Q_2-{\bar b}Q_1.
\end{array}
$$

\noindent {\rm (3)} The quadratic form $Q_4=Jac(Q_3,Q_1)=Jac(Q_2,Q_1)$ is given by
$$
\begin{array}{rcl}
Q_4&=&\left|
\begin{array}{ccc}
	dv^2 & -dudv & du^2\\
	A& B &C\\
	Ab-Ba& \frac12(Ac-Ca)&Bc-Cb 
\end{array}
\right|
\end{array},
$$
alternatively,
$$
\begin{array}{rcl}
Q_4&=&(2B^2c-2BCb-ACc+C^2a)dv^2+2(ABc-2ACb+BCa)dudv+\\
&&(2B^2a-2ABb+A^2c-ACa)du^2.
\end{array}
$$
\end{prop}

\begin{proof}
\noindent (1) Recall that we may replace $x:U\to \BR^3$ by $x+fn$ for any function $f:U\to \BR$. This replaces $Q$ by $Q+fQ_1$, and clearly $Jac(q_2+cq_1,q_1)=Jac(q_2,q_1)$ for any quadratic forms $q_1, q_2$ and constant $c$. Similarly $Q_3$ is replaced by
$$
[\a (x_u+f_un+fn_u)+\b (x_v+f_vn+fn_v),\a n_u+\b n_v.n]
$$
which clearly yields the same value. 

(2) The expression of $Q_2$ follows from the fact that it is $Jac(Q,Q_1)$.
For $Q_3$, we first observe that all coefficients of $Q_3$ vanish at points where 
$n$ is singular. Indeed, for a generic congruence, the singular set of $n$ is a regular curve 
and along it $n_v=\alpha n_u$ (or $n_v=\alpha n_u$) for some scalar function $\alpha$. Differentiating 
$[x_u, n,n]\equiv 0$ and $[x_v, n,n]\equiv 0$  along the curve proves the claim.

Suppose now that $n$ is not singular. Then $n=n_u\wedge  n_v/||n_u\wedge  n_v||$. 
We are interested in the solutions of $Q_3=0$, so do not distinguish differential forms that are non-zero multiples of each other and multiply $Q_3$ by $||n_u\wedge  n_v||$.  Then 
$$
\begin{array}{rcl}
	[ x_u, n_u, n_u\wedge  n_v]&=& x_u.( n_u\wedge( n_u\wedge  n_v))=x_u.(( n_u. n_v) n_u-( n_u. n_u). n_v))\\
	&=&Ba-A x_u. n_v=Ba-Ab+A\overline{b}\\
\end{array}
$$

Similarly,
$$\begin{array}{rcl}
	[ x_u, n_v, n_u\wedge  n_v]&=& x_u.( n_v\wedge( n_u\wedge  n_v))=x_u.(( n_v. n_v) n_u-( n_v. n_u). n_v))\\
	&=&Ca-Bb+B\overline{b},\\
\end{array}
$$
$$\begin{array}{rcl}
	[ x_v, n_u, n_u\wedge  n_v]&=& x_v.( n_u\wedge( n_u\wedge  n_v))=
	x_v.(( n_u. n_v) n_u-( n_u. n_u). n_v))\\
	&=&Bb+B\overline{b}-Ac,\\
\end{array}
$$
and
$$\begin{array}{rcl}
	[ x_v, n_v, n_u\wedge  n_v]&=& x_v.( n_v\wedge( n_u\wedge  n_v))=x_v.(( n_v. n_v) n_u-( n_v. n_u). n_v))\\
	&=&Cb+C\overline{b}-Bc.\\
\end{array}
$$
 It follows that $Q_3$ is a non-zero multiple of
$$
(Bc-bC-C{\bar b})dv^2+(Ac-Ca-2B{\bar b})dudv+(Ab-Ba-A{\bar b})du^2,
$$
and we shall take this as our $Q_3$ from now on, so that
$Q_3=Q_2-\overline{b}Q_1.$ 

Part (3) is a straightforward calculation.
\end{proof}

\begin{rem}
{\rm
\noindent Interpreting, as previously, the forms (at each point of $Z$) as elements of the real projective plane $\BR\BP^2$, note that although $Q$ is not well defined the line joining $Q$ to $Q_1$ is, and is the polar line $\widehat{Q_2}$ of $Q_2=Jac(Q,Q_1)$. Then $Q_3=Q_2-\overline{b}Q_1$ is on the line joining $Q_1$ and $Q_2$ which is the polar line $\widehat{Q_4}$ of $Q_4=Jac(Q_3,Q_1)=Jac(Q_2,Q_1)$. 
It is not hard to see that this lies on the line joining $Q$ to $Q_1$, so we have a configuration like that in Figure \ref{fig:pencilsQis}. We also get a new quadratic form $Q_5=Jac(Q_3,Q_4)$ which we deal with in more details in \S \ref{ssec:charc}.
We prove in Theorem \ref{theo:mean} (resp. Theorem \ref{theo:charDir}) that $Q_1,Q_2,Q_4$ (resp. $Q_3,Q_4,Q_4$) are vertices of a self-polar triangle.
}
\end{rem}

\begin{figure}[h]
	\begin{center}
		\includegraphics[scale=0.4]{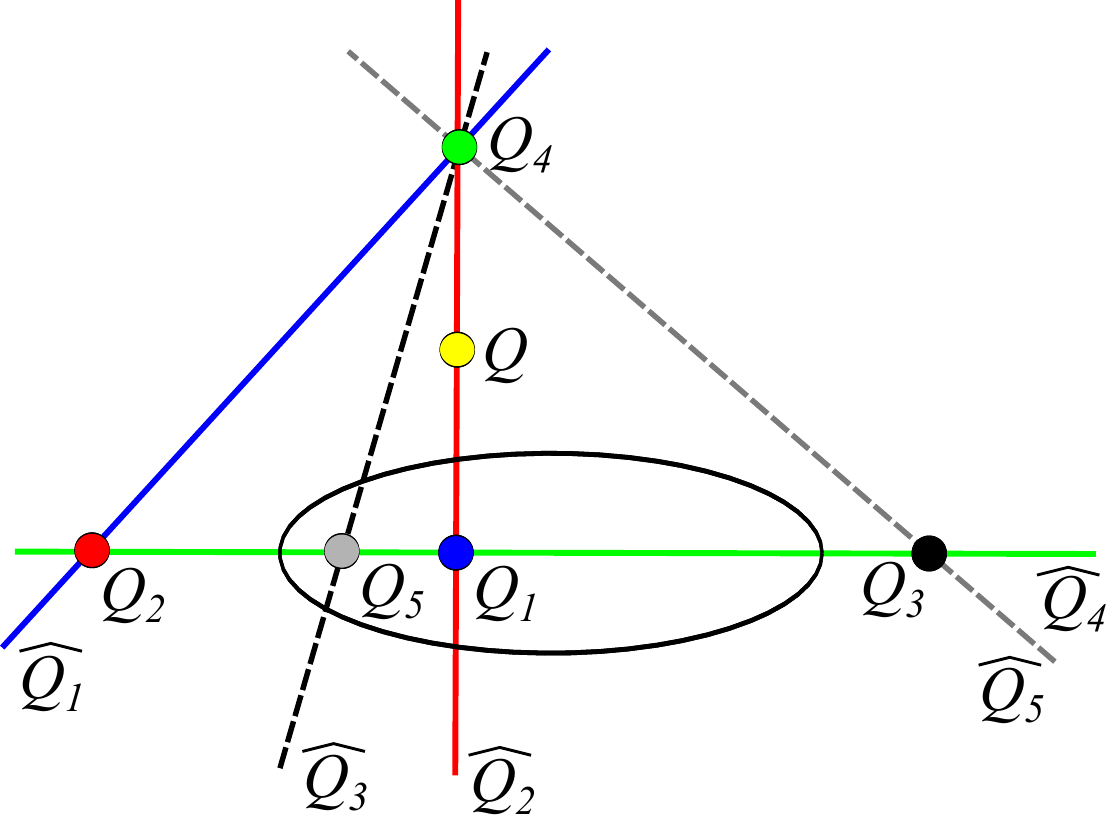}
		\caption{The configuration of the quadratic forms $Q_i$ and their polar lines.}
		\label{fig:pencilsQis}
	\end{center}
\end{figure}

As in the introduction a curve $\c(s)=(n,x)(u(s),v(s))$ through a point $z=\c(0)$ on $Z$ represents a ruled surface in $\mathbb R^3$. Along any generator of a ruled surface the central point  on $L$ is the directed distance
$$
r=-\frac{Q(\c(0))(u'(0),v'(0))}{Q_1(\c(0))(u'(0),v'(0))}
$$ 
from $x(0)$, which evidently only depends on the derivative $\c'(0)$ and $\c(0)$. As the tangent direction rotates in $T_zZ$ the central point moves up and down $L$. These ruled surfaces also have a {\it parameter of distribution} or in the terminology of Blaschke a {\it pitch} which is constant along a generator and again only depends on $\c'(0)$. The pitch of the ruled surface $(n,x)(u(s),v(s))$ at $s=0$ is 
$$
\lambda=\frac{Q_3(\c(0))(u'(0),v'(0))}{Q_1(\c(0))(u'(0),v'(0))}.
$$

We can now describe the geometry behind the quadratic forms above using Theorem \ref{theo:extrems}. 

\begin{prop}\label{prop:BDEsZ}
\noindent {\rm (1)} {\it The principal directions}: the directions in which 
the distance from the director surface to the central point has extreme value
are given by the BDE $Jac(Q,Q_1)=Q_2=0$.

\noindent {\rm (2)} {\it The torsal directions}: the directions in which the ruled surfaces are (infinitesimally) developable are given by $Q_3=0$. 

\noindent {\rm (3)} {\it The mean directions}: the pitch is $Q_3/Q_1=Q_2/Q_1-{\bar b}$ so the directions in which the pitch has its maximal values  are given by $Q_4=Jac(Q_2,Q_1)$.
\end{prop}

We determine the generic local topological configurations of the BDEs in Proposition 
\ref{prop:BDEsZ} and introduce a new BDE determining  the so-called {\it characteristic directions}.  We deal in \S\ref{ssec:LinesCurv} with the pairs determined by the geometry of the central points and in \S\ref{ssec:BDEparameter} with those determined by the pitch. These BDEs are discussed in \cite{Weatherburn}; our approach reveals the way they are related at each point via polarity in the projective plane. 
We also obtain the generic configurations of the solution curves 
of the BDEs at their singular points and on the singular set $\Sigma(n)$ of the map $n$.

As our study is local in nature, in the rest of the paper we take the line of interest $L$ to be the oriented $z$-axis and parametrise $Z$ locally at $z$ 
by $(u,v)\mapsto (n(u,v), x(u,v))$, with $(u,v)$ in a neighbourhood $U$ of the origin in $\BR^2$. 
We have $n(0,0)=(0,0,1)$ and take the directrix $x(u,v)=(x_1(u,v),x_2(u,v),0)$ with
\begin{equation}\label{3jetslambda12} 
	\begin{split}
		j^3x_1&=\alpha_{10}u+\alpha_{11}v+\alpha_{20}u^2+\alpha_{21}uv+\alpha_{22}v^2
		+\displaystyle{\sum_{i=0}^3\alpha_{3i}u^{3-i}v^i},\\
		j^3x_2&=\beta_{10}u+\beta_{11}v+\beta_{20}u^2+\beta_{21}uv+\beta_{22}v^2+
		\displaystyle{\sum_{i=0}^3\beta_{3i}u^{3-i}v^i}.
	\end{split}
\end{equation}

The map $n:U\to S^2$ is a local diffeomorphism on $U\setminus \Sigma(n)$ 
where, for generic congruences, $\Sigma(n)$ is empty or a regular curve. The map $n$ has locally a fold singularity 
at most points on $\Sigma(n)$  and a cusp singularity 
at isolated points on that curve. With the notation in Definition \ref{def:QisAndCoeff}, 
$(u,v)\in \Sigma(n)$ if and only if $(B^2-AC)(u,v)=0$, that is, $Q_1$ is degenerate.

To simplify notation, we shall omit mention of selecting $(u,v)\in U$, but implicitly we are referring to the quadratic forms and projective plane at each point $(n(u,v), x(u,v))\in Z$.

\subsection{BDEs associated to central points}\label{ssec:LinesCurv} 

As above the principle directions are given by the BDE $Q_2=0$; its integrals  curves are the
lines of principal curvature; point $z$ on $Z$ is called an umbilic point
if the all the coefficients of $Q_2$ vanish at $z$. 
We observe that umbilic points are stable on a generic congruence, that is, they persists when deforming the congruence.

Suppose that $Q_1$ is not degenerate, that is, $n$ is not singular. If 
 $\mu_1, \mu_2$ are the values described in Theorem \ref{theo:extrems}(2), the discriminant function $\delta(Q_2)$ of $Q_2$ is given by 
\begin{equation}\label{eq:discrimP}
\delta(Q_2)=4(B^2-AC)^2(H^2-K),
\end{equation}
where 
$$
\begin{array}{c}
H=\frac12(\mu_1+\mu_2),\, K=\mu_1\mu_2,
\end{array}
$$
so $\delta(Q_2)$ vanishes if and only if  $H^2-K=0$, that is precisely at umbilic points. These are points where all the coefficients of $Q_2$ vanish, that is, where $Q_1$ and $Q$ are linearly dependent. Note that because of the ambiguity in the choice of $Q$, $\mu_1$ and $ \mu_2$ are only defined up to the addition of the (same) constant, 
that is only $|\mu_1-\mu_2|$ is well defined, 
and so is $H^2-K=(\mu_1-\mu_2)^2/4$.

\begin{theo}\label{theo:princ}
{\rm (1)}  Away from $\Sigma(n)$ and umbilic points, the lines of principal curvature form a  $Q_1$-orthogonal net. The three generic configurations of BDEs of Morse Type $A_1^+$ {\rm (Figure \ref{fig:Darboux})} can occur at umbilic points on generic congruences and only these {\rm (see also \cite{CraizerGarcia})}.

{\rm (2)} For each $(u,v)\in U\setminus \Sigma(n)$ the quadratic form $Q_2\in \BR \BP^2$ is the unique point on the polar line $\widehat{Q}$ of $Q$ 
that has $Q_1$-orthogonal roots. It is the intersection of the polar lines $\widehat Q_1$  and $\widehat Q$ {\rm (}all forms evaluated at $(u,v)${\rm ).} 

{\rm (3)} On $\Sigma(n)$,  $Q_2$ factors as a product of two $1$-forms $\sigma_i$, $i=1,2$. At fold singularities of $n$ and at most points on $\Sigma(n)$, the two $1$-forms are regular and their foliations have $2$-point contact along $\Sigma(n)$. Their common tangent direction is along the kernel of $dn$, so they are transverse to $\Sigma(n)$ {\rm (Figure \ref{fig:PrincipalSingn} left)}. The pair $(\sigma_1,\sigma_2)$ is topologically equivalent to $(du, d(u-v^2))$.

At isolated fold points on $\Sigma(n)$, one of the $1$-forms is regular and the other has generically a saddle, node or focus singularity with eigenspaces when real transverse to $\Sigma(n)$  {\rm (Figure \ref{fig:PrincipalSingn} middle three figures)}.
The pair $(\sigma_1,\sigma_2)$ is topologically equivalent to 
$$
\begin{array}{ll}
(du,(v-u)du+vdv),& \mbox{if $\sigma_2$ is a saddle},\\
(du,(v+\frac18 u)du+vdv),& \mbox{if $\sigma_2$ is a node}, \\
(du,(u+v)du+vdv),& \mbox{if $\sigma_2$ is a focus}.
\end{array}
$$

At a cusp singularity of $n$,  the two $1$-forms are regular and their leaves  have $3$-point contact at the singular point of $n$ and are tangent to $\Sigma(n)$ 
{\rm (Figure \ref{fig:PrincipalSingn}, right)}. The pair $(\sigma_1,\sigma_2)$ is topologically equivalent to $(du, d(u+vu^2+v^3))$.
\end{theo}

\begin{figure}[htp]
	\begin{center}
		\includegraphics[scale=0.8]{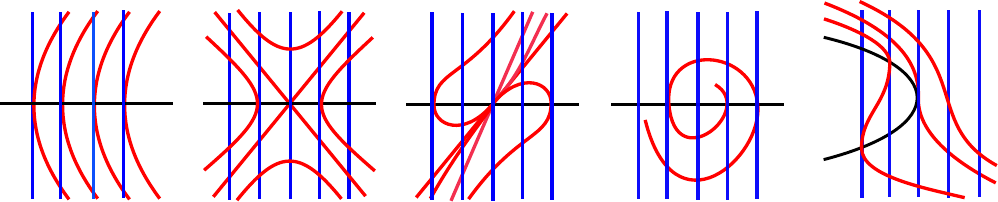}
		\caption{Principal curvature on $\Sigma(n)$: left at generic fold singular points of $n$, middle three at special fold singularities of $n$, left at a cusp singularity of $n$. The curve in black is $\Sigma(n)$.}
		\label{fig:PrincipalSingn}
	\end{center}
\end{figure}

\begin{proof}
(1) For generic congruences, the umbilic points are not on $\Sigma(n)$ (three independent conditions need to be satisfied for that to happen), 
so at umbilic points, we can presume that $n$ is non-singular, that is $Q_1$ is positive definite, 
and we can take
\begin{equation}\label{expression_n}
	n(u,v)=(u,v,\sqrt{1-u^2-v^2}).
\end{equation}

Taking $j^3x_i,i=1,2,$ as in (\ref{3jetslambda12}) we 
find that the 1-jets of the coefficients of $Q_2$ are 
$$
\begin{array}{rcl}
j^1(Bc-Cb)&=&\frac12\left(\alpha_{11} +\beta_{10}+ (\alpha_{21} + 2\beta_{20})u +
(2\alpha_{22} + \beta_{21})v\right),\\
j^1(Ac-Ca)&=& \alpha_{10}-\beta_{11}  + (2\alpha_{20} - \beta_{21})u + (\alpha_{21} - 2\beta_{22})v,\\
j^1(Ab-Ba)&=&\frac12\left(\alpha_{11} + \beta_{10} + (\alpha_{21} + 2\beta_{20})u) + (2\alpha_{22} + \beta_{21})v\right).\\
\end{array}
$$

Then the origin is an umbilic point if and only if  $\beta_{10}=-\alpha_{11}$ and $\beta_{11}=\alpha_{10}$. In that case, 
$
j^1Q_2=(a_1u+a_2v,b_1u+b_2v,-a_1u-a_2v),  
$
with
$$
\begin{array}{ll}
\begin{array}{rcl}
a_1&= &\frac12\alpha_{21}+\beta_{20},\\
b_1&=&2\alpha_{20}-\beta_{21},\\
\end{array}
&
\begin{array}{rcl}
a_2&=& \alpha_{22}+\frac{1}{2}\beta_{21},\\
b_2&=&\alpha_{21}-2\beta_{22}.
\end{array}
\end{array}
$$ 

Clearly, we can choose the 2-jets of $x_1$ and $x_2$ so  that $a_i,b_i, i=1,2$, can take any values in $\mathbb R$. It follows by the results in  \cite{bruce-fidal, bdes}  that the lines of principal curvature can have any of the three generic topological configurations lemon, star or monstar in Figure \ref{fig:Darboux}, and these are the only configurations that generically occur at umbilic points. This result on the configurations at umbilics was obtained previously by Craizer and Garcia in \cite{CraizerGarcia}.

(2) The statement follows from Theorem \ref{theo:extrems}.

(3) Suppose the origin is a fold singularity of $n$. We can take $\Sigma(n)=\{v=0\}$
and $n_v=0$ on $\Sigma(n)$. Then $n_{vv}\ne 0$, so $n_v(u,v)=vw(u,v)$  for some smooth $w$ with $w(0,0)\ne 0$ 
and $n_u,w$ linearly independent in a neighbourhood of the origin. We have
$$
\begin{array}{c}
A=n_u\cdot n_u\ne 0, B=vn_u\cdot w, C=v^2w_v\cdot w_v,\\
a=-n_u\cdot x_u, \, b=-\frac12(n_u\cdot x_v+vw\cdot x_v),\,  c=-vw\cdot x_v,
\end{array}
$$
so the coefficients of $\omega(Q_2)$ are
$$
\begin{array}{rcl}
a_P=Bc-Cb&=&-v^2\left( (n_u\cdot w)(w\cdot x_v)-\frac12(n_u\cdot x_v+vw\cdot x_u)(w\cdot w)\right),\\
b_P=Ac-Ca&=&-v\left((n_u\cdot n_u)(w\cdot x_v)-v(w\cdot w)(n_u\cdot x_u))\right),\\
c_P=Ab-Ba&=&-\frac12(n_u\cdot n_u)(n_u\cdot x_v+vw\cdot x_u)+v(n_u\cdot w)(n_u\cdot x_u).\\
\end{array}
$$
It follows that
{\footnotesize
$$
\delta(Q_2)=v^2(n_u\cdot n_u)\left[
(n_u\cdot n_u)(w\cdot x_v)^2-2(n_u\cdot w)(n_u\cdot x_u)(w\cdot x_v)+(w\cdot w)(n_u\cdot x_v)^2+vO(u,v)
\right]
$$
}
for some smooth function $O(u,v)$. As $n_u$ and $w$ are linearly independent,
it follows that 
$\overline{Q_1}(du,dv)=w\cdot w dv^2+2n_u\cdot wdudv+n_u\cdot n_udu^2$ defines (locally) a metric on $Z$. The coefficient of $v^2$ in $\delta(Q_2)$ is then 
$\Lambda=(n_u\cdot n_u)\overline{Q_1}(w\cdot x_v,-n_u\cdot x_v)$. It vanishes if and only if 
$w\cdot x_v=n_u\cdot x_v=0$, which does not occur on generic congruences at points on $\Sigma(n)$. Therefore, $\Lambda>0$.
We now have two cases, $c_P(0,0)=(n_u\cdot n_u)(n_u\cdot x_v)$ vanishing or not.

If $n_u\cdot x_v(0,0)\ne 0$, $c_P(0,0)\ne 0$ so $Q_2$  is a non-zero multiple of the product of the following two regular 1-forms
$$
\sigma_i=2c_Pdu+(-b_P-(-1)^i\sqrt{\delta(Q_2)}\,)dv,\, i=1,2.
$$  
On $v=0$, the above one forms reduce to $du=0$ which is the kernel direction of $dn$. Their leaves have $2$-point contact at points on $v=0$. The result on their topological configurations follows  from Theorem 2.2 in \cite{davbook}.

If $n_u\cdot x_v(0,0)=0$, then all the coefficients of $Q_2=0$  vanish at the origin.
Away from the curve $c_P(u,v)=0$, $Q_2$ is still a product of the above 1-forms
$\omega_i$. 
Writing $a_P=v^2\bar{a}_P$, $b_P=v\bar{b}_P$, $c_P=\bar{c}_P$, 
and assuming $\bar{b}_P(0,0)\ne 0$ for generic congruences (we have already two independent condition at the origin), and without loss of generality that $\bar{b}_P(0,0)>0$,
we have 
$$
\sigma_i=
2\bar{c}_Pdu+
v(-\bar{b}_P-(-1)^i\bar{b}_P(1-\frac{4\bar{a}_P\bar{c}_P}{\bar{b}_P^2})^{\frac12}\,)dv,\, i=1,2.
$$

The 1-form  
$\sigma_1=\bar{c}_P(2du+h(u,v)dv)$ for some regular function $h$. So it defines a regular foliation away from $\bar{c}_P=0$ which extends to a regular foliation on $\bar{c}_P=0$ (we get a line of singularities on $\bar{c}_P=0$ which does not alter the configuration of the foliation determined by $\sigma_1$). 
For the analysis of the singularity of the 1-form $\sigma_2$ we can take the first two component of $n$ in the form  $(u,(1+k_{10}u+k_{11}v+O(2))v^2+l_2u^2+O(u^3)))$. Then $n_u\cdot x_v(0,0)=\alpha_{11}=0$ and 
$j^1\sigma_2=\left((2l_2\beta_{11} + \alpha_{21})u) + 2(\alpha_{22} + \beta_{10})v\right)du+2\beta_{11}vdv.$ Clearly, $\sigma_1$ can have generically a saddle, node or focus singularity. The result on the topological configurations of the pair $(\sigma_1,\sigma_2)$ follows from  Theorem 2.2 in \cite{davbook}.

At a cusp singularity of $n$, we can still take $n_v=0$ on $\Sigma(n)=\{\lambda(u,v)=0\}$ for some smooth function 
$\lambda$ with $\lambda_u(0,0)\lambda_{vv}(0,0)\ne 0$. We can then proceed as above and write $n_v(u,v)=\lambda(u,v)w(u,v)$ with $w(0,0)\ne 0$ and $x_u$ and $w$ linearly independent. The coefficients $a_P,b_P,c_P$ are as above but 
replacing the factor $v$ by $\lambda(u,v)$. 
For generic congruences we cannot have any extra conditions, so $c_P(0,0)\ne 0$ and $Q_2$ factors as a product of two $1$-forms $\sigma_i$, $i=1,2$. We can show that the leaves at the origin have $3$-point contact between them and $2$-point contact with $\Sigma(n)$.
The result on the topological configurations of the pair $(\sigma_1,\sigma_2)$ also follows from Theorem 2.2 in \cite{davbook}.
\end{proof} 


\subsection{BDEs associated to the parameter distribution}\label{ssec:BDEparameter}

We consider as in \S\ref{ssec:LinesCurv} a curve $\gamma$ on the surface $Z$ with $\gamma(0)=(n(0), x(0))= z$ and seek the directions $\gamma'(0)\in T_zZ$ for which the parameter distribution $\lambda$ of the associated ruled surface in $\BR^3$ vanishes. These are the  {\it torsal directions}  
and their integral curves are  the {\it torsal curves}. 
We have the following observation.

\begin{prop}\label{prop:zerocoeffTorsal}
For a generic congruence, the coefficients of the BDE $Q_3=0$ of the torsal curves vanish simultaneously at $z$ if and only if $n$ is singular at $z$.
\end{prop}

\begin{proof}
We showed in the proof of Proposition \ref{prop:BDEscong} (2) that the coefficients of $Q_3$ all vanish on $\Sigma(n)$. 
For the converse, setting to zero the coefficients of $Q_3$ leads to $(AC-B^2)(Ac-aC)=0$. The first factor vanishes if and only if $n$ is singular. 
 If the second factor vanishes then it is easy to see that $\bar{b}=0$ and $z$ an umbilic point,  
so we have three independent conditions which cannot be satisfied simultaneously  for a  generic congruence.
\end{proof}

\begin{theo}\label{theo:lambda=0}
{\rm(1)} 
At each point $z$ on $Z\setminus \Sigma(n)$ there are $2,1$ or $0$ torsal directions in $T_zZ$ for which the pitch $\lambda$ vanishes. These are given by the BDE $Q_3=Q_2-{\bar b} Q_1=0$.
	
We have  $\delta(Q_3)=\delta(Q_2)+\overline{b}^2\delta(Q_1)$. The discriminant of the BDE $Q_3=0$ is generically a smooth curve on  $Z$, called the parabolic curve, and all the configurations in {\rm Figure \ref{fig:Dav}} can occur, and generically only these.
The set on $Z$ where $\delta(Q_3)>0$ {\rm (}resp.  $\delta(Q_3)< 0${\rm )}  is called the {\rm hyperbolic (}resp. {\rm elliptic) region} of $Z$.

{\rm(2)}  The torsal curves extend to $\Sigma(n)$. For generic congruences, the curve $\Sigma(n)$ lies in the closure of the hyperbolic region of $Z$.
The parabolic curve intersect $\Sigma(n)$ tangentially at isolated points. At such points, the torsal curves form  locally a family of cusps as in {\rm Figure \ref{fig:Torsal_Par}}. 
\end{theo}

\begin{figure}
	\begin{center}
		\includegraphics[scale=0.8]{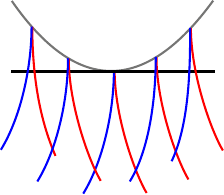}
		\caption{Generic configuration of the torsal curves on $\Sigma(n)$ and at a parabolic point on $\Sigma(n)$. The thick curve in black represents $\Sigma(n)$ and the curve in grey the parabolic curve.}
		\label{fig:Torsal_Par}
	\end{center}
\end{figure}

\begin{proof}
(1) We have already found the form of the BDE which is also $Q_2-\overline{b}Q_1=0.$ This means that $Q_3$ belongs to the 
pencil determined by $Q_1$ and  $Q_2$, that is, to the polar line of $Jac(Q_1,Q_2).$ 

A straightforward calculation shows that  $\delta(Q_3)=\delta(Q_2)+4\overline{b}^2\delta(Q_1)$. 
Using (\ref{eq:discrimP}), we get $\delta(Q_3)=4\delta(Q_1)^2(H^2-K)+4\overline{b}^2\delta(Q_1)$, with $\delta(Q_1)=B^2-AC$, so 
\begin{equation}\label{eq:zeros)mega_F}
\delta(F)=0\iff H^2-K+\frac{\overline{b}^2}{B^2-AC}=0.
\end{equation}

Equation (\ref{eq:zeros)mega_F}) generically determines a regular curve on the surface. 
For the configuration of the BDE $Q_3=0$, we use the setting of the proof of Theorem \ref{theo:princ}(1). 
Then the 2-jets of the coefficients $(a_T,b_T,c_T)$ of $Q_3$ are given by
$$
\begin{array}{rl}
j^2a_T=&\alpha_{11} + \alpha_{21}u + 2\alpha_{22}v +\alpha_{31}u^2 - (\beta_{11} - 2\alpha_{32})uv + (3\alpha_{33} + \alpha_{11})v^2,\\
j^2b_T=& \alpha_{10}-\beta_{11}  + (2\alpha_{20} - \beta_{21})u + (\alpha_{21} - 2\beta_{22})v +(3\alpha_{30}-\beta_{11} - \beta_{31} )u^2 +\\
& ( \alpha_{11}+ 2\alpha_{31}-2\beta_{32}  -\beta_{10} )uv +(\alpha_{10} + \alpha_{32}-3\beta_{33} )v^2,\\
j^2c_T=&-\beta_{10} - 2\beta_{20}u - \beta_{21}v  -(3\beta_{30} + \beta_{10})u^2 +( \alpha_{10}-2\beta_{31} )uv- \beta_{32}v^2.
\end{array}
$$

The origin is on the discriminant if and only if  $(\alpha_{10}-\beta_{11})^2+4\alpha_{11}\beta_{10}=0$. 
It is a singular point of $Q_3=0$ if and only if 
{\small
\begin{equation}\label{eq:singFlatBDE}
(\alpha_{10}-\beta_{11} )^3\beta_{20} + 2(\alpha_{10}-\beta_{11} )^2(\beta_{21} - \alpha_{20})\beta_{10} + 4\beta_{10}^2(\alpha_{10}-\beta_{11} )
(\beta_{22} - \alpha_{21}) - 8\alpha_{22}\beta_{10}^3=0.
\end{equation}
}
Following the arguments in \S\ref{sec:prel},  we can show that 
all the stable singularities in \mbox{Figure \ref{fig:Dav}} can occur. For generic congruences, those are the only singularities that can occur.

(2) We follow the setting of the proof of Theorem \ref{theo:princ} (3) at a fold or a cusp singularity of $n$. 
We set $n_v(u,v)=\lambda(u,v)w(u,v)$ with $\Sigma(n)=\{\lambda(u,v)=0\}$ and $w,n_u$ linearly independent. 
Then $(a_T,b_T,c_T)=(\lambda^2\tilde{a}_T,\lambda\tilde{b}_T,\lambda\tilde{c}_T)$ for some smooth functions $\tilde{a}_T,\tilde{b}_T,\tilde{c}_T$. 
This means that the torsal curves have a removable singularity on $\Sigma(n)$. We define their extension to $\Sigma(n)$ as the solutions of the BDE with coefficients
$(\lambda \tilde{a}_T,\tilde{b}_T,\tilde{c}_T)$. 
Then the parabolic curve extended to points on $\Sigma(n)$ 
and is given by $\tilde{b}_T^2-4\lambda \tilde{a}_T\tilde{c}_T=0$. In particular, it intersects $\Sigma(n)$ at the origin if and only if $\tilde{b}_T(0,0)=0$ and for generic congruences, the two curves have ordinary tangency at that point.

The remaining part of the proof follows by computing the initial jets of the coefficients 
 $(\lambda\tilde{a}_T,\tilde{b}_T,\tilde{c}_T)$. 
 Following the setting 
in the proof of Theorem \ref{theo:princ} (3) at a fold singularity of $n$, we get
$$
\begin{array}{rcl}
j^1(v\tilde{a}_T)&=&4\alpha_{11}v,\\
j^1\tilde{b}_T&=&-2\beta_{11} + 2(2l_2\alpha_{11} - \beta_{11}k_{10} - \beta_{21})u+ 
(4\alpha_{10} -3\beta_{11}k_{11} - 4\beta_{22})v,\\
j^1\tilde{c}_T&=&-2\beta_{10} + 2(2l_2\alpha_{10} - \beta_{10}k_{10} - 2\beta_{20})u - (3\beta_{10}l_{11} + 2\beta_{21})v.
\end{array}
$$
The origin is on the parabolic curve if and only if $\beta_{11}=0$. The  parabolic set is generically a smooth curve ($\alpha_{11}\beta_{10}\ne 0$) and has ordinary tangency with $\Sigma(n)$ ($2l_2\alpha_{11} - \beta_{11}k_{10} - \beta_{21}\ne 0$) at the origin.
The torsal curves form a family of cusps along the parabolic curve as $\alpha_{11}\beta_{10}\ne 0$. Observe that along $\Sigma(n)$, $\delta(Q_3)=\tilde{b}_T^2(u,0)>0$, so there are always two real torsal directions away from parabolic points.

The claim at a cusp singularity of $n$ follows similarly observing that for generic congruences the cusp singularity of $n$ is not on the parabolic curve.
\end{proof}

We turn now to the extremes of the pitch, these occur along the mean directions.

\begin{theo}\label{theo:mean}
{\rm (1)} At every non-umbilic point on $Z\setminus \Sigma(n)$ there are two $Q_1$-orthogonal directions along which the parameter distribution has an extremum. These are given by the BDE $Q_4=Jac(Q_2,Q_1)=0$, that is, 
\begin{equation}\label{eq:Jac(P,g)}
Q_4=
\left|
\begin{array}{ccc}
	dv^2&-dudv&du^2\\
	A & B & C\\
Ab-Ba&(Ac-Ca)/2&Bc-Cb
\end{array}
\right|=0.
\end{equation}

The discriminant of the BDE {\rm (\ref{eq:Jac(P,g)})} consists of the umbilic points, and all the configurations in {\rm Figure \ref{fig:Darboux}} can occur at such points and only these.

{\rm (2)} The triple $Q_1,Q_2,Q_4$ are vertices of a self-polar triangle. 

{\rm (3)}  For generic congruences, the solution curves of $Q_4=0$ extend to $\Sigma(n)$. Their configurations are as those in {\rm Figure \ref{fig:PrincipalSingn}} 
with the special fold points generically distinct from those of the principal curves.
\end{theo}

\begin{proof}
(1) We have already established the form of the BDE, and the orthogonality condition follows from Theorem \ref{theo:extrems}.
We have $\delta(Q_4)=-\delta(Q_1)\delta(Q_2)$ which vanishes only at umbilic points or on $\Sigma(n)$.

At an umbilic point which is generically not on $\Sigma(n)$, and with the same setting as for the previous cases, we have $j^1Q_4=(\bar{a}_1u+\bar{a}_2v,\bar{b}_1u+\bar{b}_2v,-\bar{a}_1u-\bar{a}_2v)$ with
$$
\begin{array}{lcl}
\bar{a}_1=-2\alpha_{20} + \beta_{21}&\quad&\bar{a}_2=-\alpha_{21} +2\beta_{22},\\
\bar{b}_1=2(\alpha_{21} + 2\beta_{20})&&b_2=2(2\alpha_{22} + \beta_{21}).
\end{array}
$$

Clearly, as in the proof of Theorem \ref{theo:princ},  we can obtain the three generic configurations in Figure \ref{fig:Darboux} and show that only these occur.
Observe that $\bar{a}_1u+\bar{a}_2v=-b_1u-b_2v$ and $\bar{b}_1u+\bar{b}_2v=4(a_1u+a_2v)$
with  $j^1Q_2=(a_1u+a_2v,b_1u+b_2v,-a_1u-a_2v)$ as in the proof of Theorem \ref{theo:princ}.

(2) We have $Jac(Q_2,Q_4)=\delta(Q_2)Q_1$, $Jac(Q_1,Q_4)=\delta(Q_1)Q_2$ 
and $Q_4=Jac(Q_1,Q_2)$, so $Q_1,Q_2,Q_4$ are vertices of a self-polar triangle.

(3) The analysis at points on $\Sigma(n)$ is identical to that for principal curves and is omitted. 
\end{proof}

\subsection{Characteristic directions and characteristic points}\label{ssec:charc}

The BDEs of the torsal and mean directions determine another BDE on the surface, namely, $Q_5=Jac(Q_3,Q_4)=0$. We call its solutions at each point $z$ the {\it characteristic directions} and their integral curves the {\it characteristic curves}. 

\begin{theo}\label{theo:charDir}
{\rm (1)} At each point $z$ on the surface $Z\setminus \Sigma(n)$ there are $2,1$ or $0$ characteristic directions in $T_zZ$. These are given 
by the BDE
$$
Q_5=-\delta(Q_2)Q_1-\overline{b}\delta(Q_1) Q_2=0,
$$
so $Q_5$ is on the polar line $\widehat{Q_4}$ of $Q_4$. The triple $Q_3,Q_4,Q_5$ are vertices of a self-polar triangle.

{\rm (2)}  The discriminant of the BDE $Q_5=0$ is given by  $\delta(Q_5)=\delta(Q_1)\delta(Q_2)\delta(Q_3)$. 
Consequently,  when the torsal directions are real the characteristic ones are imaginary and vice-versa, that is, the characteristic curves lie on the closure of the elliptic region of $Z$. 

{\rm (3)} The discriminant of $Q_5$ is the union of the parabolic curve together with the umbilic points. The folded singularities of the 
characteristic directions occur at the same points as those of the torsal directions. 
The two configurations have opposite indices (when one has a folded saddle, the other is a folded node or focus). At umbilic points, the characteristic curves have the same configurations as those of the principal curves.

{\rm (4)} The characteristic curves BDE extend to $\Sigma(n)$. 
For generic congruences and away from parabolic points, there are no characteristic directions on $\Sigma(n)$. At a parabolic point on 
$\Sigma(n)$, we get a family of cusps along the parabolic curve.
\end{theo}

\begin{proof}
(1) and (2). The expression for $Q_5$ follows by calculating $Jac(Q_2,Q_3)$.
Clearly it is on the polar line $\widehat{Q_4}$ of $Q_4=Jac(Q_1,Q_2)$. 
We have $Jac(Q_4,Q_5)=\delta(Q_1)\delta(Q_2)Q_3$ so the polar line $\widehat{Q_3}$ of $Q_3$ contains $Q_4$ and $Q_5$. The polar line $\widehat{Q_5}$ of $Q_5$ contains $Q_3$ and $Q_4$ by definition of $Q_5.$ This proves that 
$Q_3,Q_4,Q_5$ are vertices of a self-polar triangle.

  Observe that away from $\Sigma(n)$ and umbilic points $\delta(Q_1)<0$ and $\delta(Q_2)>0$, so $\delta(Q_5)$ and $\delta(Q_3)$ have opposite signs. It follows that the torsal curves and the characteristic curves live on opposite sides of the parabolic curve.

(3) We follow the setting of the proof of the previous cases and compute the 2-jet of the coefficients of BDE $Q_5$. At a parabolic point, we find that the BDE is singular if and only if the condition  (\ref{eq:singFlatBDE}) is satisfied. 
We compute the scalar $\lambda$  that determines the type of the folded singularity (see \S \ref{sec:prel}) and find that it is the opposite of that associated with of the folded singularity of the torsal direction.

At umbilic points, the 1-jet of $Q_5$ is a scalar multiple of that of the principal curves, so at generic umbilic points, the two BDEs have the same configurations.
\end{proof}

When the characteristic directions are real at a point $z$ on $Z$, they determine points on the line $L\subset \BR^3$ associated to $z$ 
which we call {\it characteristic points}. We call the surface in $\mathbb R^3$ that they trace as $z$ varies in $Z$ the {\it characteristic surface}. 

\begin{theo}
There are $2,1$ or $0$ characteristic points on the line $L\subset \BR^3$. 
When there are two of them, there are no focal points and vice-versa.
The midpoint of the characterise points is the midpoint of the boundary points. The characteristic points lie between the boundary points; see 
 {\rm Figure \ref{fig:position}}.  
\end{theo}

\begin{proof}
We consider the case when $n$ is not singular; the singular case is similar. Then
each characteristic direction (when it exists) at a point $z\in Z$ determines a  characteristic point on the line $L\subset \mathbb R^3$ with distance $r=-x'.n'/n'.n'$ from $x$, with $(u',v')$ a solution of $Q_5=0$. 
Eliminating  $u', v'$ gives
$$
r^2-2Hr+K+\frac{AC-B^2}{\bar{b}^2}(H^2-K)^2=0.
$$

It follows that the midpoint $(x+Hn)(0,0)$  of the characteristic points  is the midpoint of the boundary points (these are 
given by $
r^2-2Hr+K=0,
$ \cite{Weatherburn}).
The square of the distance between  the characteristic points is  $H^2-K-(AC-B^2)(H^2-K)^2/{\bar{b}^2}$
and that between the boundary points is $H^2-K$. Therefore the characteristic points are between the boundary points.
\end{proof}

\begin{figure}[htp]
	\begin{center}
		\includegraphics[scale=0.4]{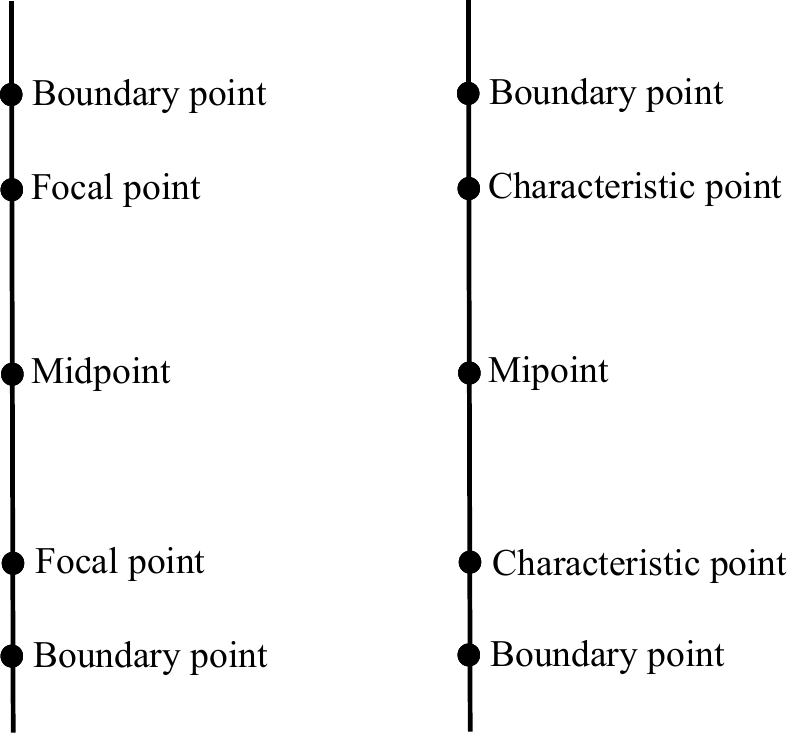}
		\caption{Relative position of the boundary, focal/characteristic and midpoints.}
		\label{fig:position}
	\end{center}
\end{figure}

\section{Normal congruences}
A congruence $Z\subset \BL$ is said to be a normal congruence if for some smooth surface $X\subset \BR^3$ the set $Z$ consists of the normals to $X$ denoted by $N(X)$.

Normal congruences are important in geometrical optics, where $X$ can be thought of as a wavefront and the associated lines as rays. In this case the quadratic forms in the previous section are determined by the first and second fundamental forms of $X$. 

We parametrise the directrix $X$ away from its umbilic points so that the coordinate curves represent the lines of principal curvature. 
Then the coefficients of the first fundamental form are 
$x_u\cdot x_u=E$, $x_u\cdot x_v=0$, $x_v\cdot x_v=G$.

The quadratic form $Q_1$ is the third fundamental form of $X$ and has coefficients
$A=\kappa_1^2E$, $B=0$ and $C=\kappa_2^2G$.  

For normal congruences, $\bar{b}=0$ (in fact $\bar{b}=0$ if and only if $Z$ is a normal congruence, \cite{Weatherburn}, p. 196) 
and $a=\kappa_1E,b=0,c=\kappa_2G$ 
are the coefficients of the second fundamental form of $X$.

The quadratic form $Q_2=\kappa_1\kappa_2EG(\kappa_1-\kappa_2)dudv$, so 
the solution curves of $Q_2=0$, i.e., the principal curves of the congruence $Z$ coincide with the lines of principal curvature of $X$ in the parameter space.
In fact, the natural map $X\to \BL, x\mapsto (x, N(x))$ with image $N(X)$, takes the classical principle directions of $X$ to the principal directions of the congruence $N(X)$. (Here $\Sigma(n)$ is the parabolic set of $X$  given by $\kappa_1\kappa_2=0$, so the principal directions on $Z$ have removable singularities on $\Sigma(n)$.)

We have $Q_3=Q_2$, so the torsal curves coincide with the principal curves, so 
every non-umbilic point on $N(X)$ is a hyperbolic point.

The BDE of the mean directions is given by  
$$
Q_4=Jac(Q_1,Q_2)=\kappa_2^2Gdv^2-\kappa_1^2Edu^2=0.
$$  

The directions determined by $Q_4$ are called {\it minimal orthogonal spherical image directions}  in \cite{joey}. They are the unique pair of tangent directions to $X$ that have orthogonal images under $dn$ and that are inclined at a minimal angle at each point of $X$.

The characteristic BDE is $
Q_5=Jac(Q_4,Q_2)=\kappa_2^2Gdv^2+\kappa_1^2Edu^2=0,
$  
so does not have real solutions (there are no elliptic points on $N(X)$).

\begin{acknow}
	The work in this paper was partially supported by the FAPESP Thematic project grant 2019/07316-0.
\end{acknow}


\noindent 
JWB: Department of Mathematical Sciences, University of Liverpool, Liverpool, L69 3BXl\\
E-mail: billbrucesingular@gmail.com\\

\noindent
FT: Instituto de Ci\^encias Matem\'aticas e de Computac\~ao - USP, Avenida Trabalhador s\~ao-carlense, 400 - Centro,
CEP: 13566-590 - S\~ao Carlos - SP, Brazil.\\
E-mail: faridtari@icmc.usp.br

\end{document}